\documentclass[12pt]{article}
\usepackage[british]{babel}
\usepackage{amsmath}
\usepackage{amsthm}
\usepackage{amsfonts}
\usepackage{amssymb}
\usepackage{fancyhdr}
\usepackage[all]{xy}           % https://en.wikibooks.org/wiki/LaTeX/Xy-pic
\usepackage{graphicx}
\usepackage[left=3cm, right=3cm]{geometry}
\usepackage{enumerate}
\usepackage{cite}

\theoremstyle{plain}
\newtheorem{teo}{Theorem}[section]
\newtheorem{thm}[teo]{Theorem}
\newtheorem{cor}[teo]{Corollary}
\newtheorem{lem}[teo]{Lemma}
\newtheorem{remark}[teo]{Remark}
\newtheorem{pro}[teo]{Proposition}

\newtheorem{examp}[teo]{Example}

\def\M{\text{\rm M}}

\def\deg{\text{\rm deg}}

\def\dim{\text{\rm dim}}

\newcommand{\overbar}[1]{\mkern 1.5mu\overline{\mkern-1.5mu#1\mkern-1.5mu}\mkern 1.5mu}

\author{Tsiu-Kwen Lee$^\dag$ and Jheng-Huei Lin$^\ddag$}

\title{\textbf{Certain functional identities on division rings}}
\date{}

\begin{document}

\maketitle

\centerline {Department of Mathematics, National Taiwan
University$^\dag$, $^\ddag$}

\centerline {Taipei 106, Taiwan}

\centerline {tklee@math.ntu.edu.tw$^\dag$ and r01221012@ntu.edu.tw$^\ddag$}

\vskip8pt

\begin{abstract}
\noindent
We study the functional identity $G(x)f(x)=H(x)$ on a division ring $D$, where $f \colon D\to D$ is an additive map and $G(X)\ne 0, H(X)$ are generalized polynomials in the variable $X$ with coefficients in $D$. Precisely, it is proved that either $D$ is finite-dimensional over its center or $f$ is an elementary operator. Applying the result and its consequences, we prove that if $D$ is a noncommutative division ring of characteristic not $2$,
then the only solution of additive maps $f, g$ on $D$ satisfying the identity
$f(x) = x^n g(x^{-1})$ with $n\ne 2$ a positive integer is the trivial case, that is, $f=0$ and $g=0$.
This extends Catalano and Merch\'{a}n's result in 2023  to get a complete solution.
\end{abstract}

{ \hfill\break \noindent 2020 {\it Mathematics Subject Classification.}\ \  16R60, 16R50, 16K40.   \vskip4pt

% 16R50 Other kinds of identities (generalized polynomial, rational, involution)
% 16R60 Functional identities (associative rings and algebras)
% 16K40 Infinite-dimensional and general division rings
% 16S99 None of the above, but in this section
%    16S Associative rings and algebras arising under various constructions
% 16U60 Units, groups of units (associative rings and algebras)

   \vskip4pt

\noindent Corresponding author:\ \ Jheng-Huei Lin.   \vskip12pt

\section{Introduction}

  Throughout, rings or algebras are associative with unity. The aim of this paper is to study certain functional identities on division rings.
  The study of identities on division rings involving inverses can be traced back to the research of the Cauchy's functional equation $f(x+y)=f(x)+f(y)$ in 1821.
  In 1963, Halperin raised a question of whether an additive map $f \colon \mathbb{R} \rightarrow \mathbb{R}$ satisfying the identity $f(x)=x^2 f(x^{-1})$ is continuous or not (see \cite{Aczel1964}).
  In 1964, Kurepa \cite{Kurepa1964} proved that if $f,g \colon \mathbb{R} \rightarrow \mathbb{R}$ are nonzero additive maps, satisfying the identity $f(t) = P(t)g(t^{-1})$, where $P \colon \mathbb{R} \rightarrow \mathbb{R}$ is continuous with $P(1)=1$, then $P(t) = t^2$, $f(t)+g(t)=2tg(1)$, and the map $t \mapsto f(t)-tf(1)$ is a derivation.
  In 1968, Nishiyama and Horinouchi \cite{Nishiyama1968} studied the identity $ f(x^n) = ax^k f(x^m)$, where $n,k,m$ are integers, on real numbers.
  In 1970--1971, Kannappan and Kurepa \cite{Kannappan1970,Kannappan1971} studied several identities with inverses on real numbers.
  In 1987, Ng \cite{Ng1987} completely characterized the identity $F(x) + M(x)G(1/x)=0$, where $F,G$ are additive and $M$ is multiplicative, on a field of characteristic not $2$.
  In the same year, Vukman \cite{Vukman1987} proved that if $f$ is an additive map on a noncommutative division ring of characteristic not $2$ satisfying the identity $f(x)=-x^2 f(x^{-1})$, then $f=0$.
  In 1990, Bre\v sar and Vukman \cite{Bresar1990} investigated the identity $f(x)=x^2 g(x^{-1})$ on the algebra of all bounded linear operators of a Banach space.
  We refer the reader to \cite{Ebanks2015,Ebanks2017} for related history.

Our present study is also related to the notion of derivations. By a derivation of a ring $R$ we mean an additive map $d \colon R \rightarrow R$ satisfying $d(xy)=d(x)y+xd(y)$ for all $x,y \in R$. All derivations of a ring $R$ satisfy the identity $f(x)= -x f(x^{-1})x$ for all invertible elements $x\in R$.
  In 2018, Catalano \cite{Catalano2018} studied a more general form $f(x)x^{-1} + xg(x^{-1})=0$ on division rings and on matrix rings over division rings.
  She also raised a question of whether the condition ``${\text{\rm char}}\,D \ne 3$" can be dropped in her theorem \cite[Theorem 4]{Catalano2018}.
  In 2020, Arga\c{c}, Ero{\v g}lu, and the authors \cite{Argac2020} generalized her results and answered the question in the affirmative.
  More precisely, they proved that if $R$ is either a noncommutative division ring or a matrix ring $\M_n(D)$, $n \geq 1$, over a division ring $D$ with ${\text{\rm char}}\,D \ne 2$, and if $f , g\colon R\to R$ are additive maps satisfying the identity $f(x)x^{-1} + xg(x^{-1})=0$, then $f(x) = xq + d(x)$ and $g(x) = -qx + d(x)$, where $d$ is a derivation of $R$ and $q:=f(1)$.
  In 2023, Dar and Jing \cite{Dar2023} showed that if $f,g$ are additive maps on $R$, which is either a noncommutative division ring of characteristic not $2$ or a matrix ring $\M_n(D)$, $n > 1$, over a noncommutative division ring $D$ with ${\text{\rm char}}\,D \ne 2,3$, satisfying the identity $f(x)=-x^2 g(x^{-1})$, then $f(x) = xq$ and $g(x) = -xq$ for some $q \in R$.

  In a recent paper \cite{Catalano2023}, Catalano and Merch\'{a}n studied the identity
$  f(x)=-x^n g(x^{-1})$
  on a division ring, where $f , g$ are additive maps and $n$ is a positive integer. They obtained the following result for the cases $n=3$ and $n=4$. For a ring $R$, we denote $R^{\times}$ to be the set of all units in $R$.

  \begin{thm} \label{thm1} {\rm (\cite[Theorem 1]{Catalano2023})}
    Let $D$ be a division ring and $f , g \colon D\to D$ be additive maps. Then $f=g=0$ if one of the following holds:
    \begin{enumerate} [{\rm (i)}]
      \item ${\text{\rm char}}\,D \ne 2,3$ and $f(x)=-x^3 g(x^{-1})$ for all $x \in D^{\times}$;
      \item ${\text{\rm char}}\,D \ne 2,3$ and $f(x)=-x^4 g(x^{-1})$ for all $x \in D^{\times}$.
    \end{enumerate}
  \end{thm}

  We remark that, by considering $f$ and $-g$, it is easy to see that studying the identity $f(x)=-x^n g(x^{-1})$ is equivalent to studying the identity
  \begin{eqnarray} \label{eq:17}
    f(x) = x^n g(x^{-1})
  \end{eqnarray}
  where $f,g$ are additive maps and $n $ is a positive integer. Note that the identity Eq.\eqref{eq:17} on fields has been completely solved by Ng (see \cite[Theorem 4.1]{Ng1987}) because the map $x \mapsto x^n$ is multiplicative in this case.

  Dar and Jing have solved the case $n=2$ in Eq.\eqref{eq:17} as follows.

 \begin{thm}(\cite[Theorem 4.1]{Dar2023})
 Let $D$ be a division ring, which is not a field, with characteristic
different from $2$ and let $f,  g\colon D\to D$ be additive maps satisfying the identity
$f(x)+x^2g(x^{-1}) = 0$ for all $x\in D^{\times}$. Then $f(x) = xq$ and $g(x) =-xq$ for all $x\in D$,
where $q\in D$ is a fixed element.
 \label{thm13}
  \end{thm}

In the paper we study certain functional identities involving maps defined by generalized polynomials in the variable $X$ with coefficients in a division ring $D$ with center $Z(D)$.
  Let $D\{X\}$ denote the free product of $Z(D)$-algebra $D$ and the polynomial algebra $Z(D)[X]$ over $Z(D)$. Precisely, a monomial in $D\{X\}$  of degree $s$ is of the form
  $$
  a_1Xa_2Xa_3\cdots a_sXa_{s+1}
  $$
  for some $a_i\in D$. An element of $D\{X\}$ is the sum of finitely many monomials. We refer the reader to \cite{Martindale1969} or \cite{Chuang1988}.\vskip6pt

\noindent {\bf Definition.}\ \  A map $f$ from $D$ into itself is called an elementary operator if there exist finitely many $a_i, b_i\in D$ such that $f(x)=\sum_{i}a_ixb_i$ for all $x\in D$. \vskip6pt

We will prove the following result.\vskip6pt

\noindent {\bf Theorem A.} (Theorem \ref{thm11}) {\it Let $D$ be a division ring, and let $f\colon D\to D$ be  an additive map
  satisfying $G(x)f(x)=H(x)$ for all $x\in D$, where $G(X), H(X)\in D\{X\}$.
 If $G(X)\ne 0$, then either $D$ is finite-dimensional over $Z(D)$ or $f$ is an elementary operator.}\vskip6pt

Applying such a characterization and its consequences, we solve Eq.\eqref{eq:17}, which extends Catalano and Merch\'{a}n's result in 2023  to get a complete solution.  \vskip6pt

\noindent {\bf Theorem B.} (Theorem \ref{thm10}) {\it Let $D$ be a noncommutative division ring with ${\text{\rm char}}\,D \ne 2$. Let $f, g \colon D \rightarrow D$ be additive maps satisfying
  $f(x) = x^n g(x^{-1})$
  for all $x \in D^{\times}$, where $n\ne 2$ is a positive integer. Then $f = 0$ and $g = 0$.}\vskip6pt

\section{The identity $G(x)f(x)=H(x)$}
Throughout, let $D$ be a division ring with center $Z(D)$.
We define $D\{X_1,\ldots,X_m\}$ to be the free product of $Z(D)$-algebras $D$ and the free algebra $Z(D)\{X_1,\ldots,X_m\}$ over $Z(D)$ or
the generalized free algebra over its center $Z(D)$ in the variable $X_1,\ldots,X_m$ with coefficients in $D$
(see \cite{Martindale1969} or \cite{Chuang1988}).
We say that $D$ is a GPI-algebra if there exists a nonzero $g(X_1,\ldots,X_s)\in D\{X_1,\ldots,X_s\}$ such that $g(x_1,\ldots,x_s)=0$ for all $x_i\in D$.
In this case, we say that $g(X_1,\ldots,X_s)$ is a nontrivial GPI for $D$.
The following is a special case of \cite[Theorem 3]{Martindale1969}.

\begin{thm} (Martindale 1969)\label{thm12}
Every division GPI-algebra is finite-dimensional over its center.
\end{thm}

Given additive maps $f_{i1},\ldots,f_{is}$ from $D$ into itself and $a_{ij}\in D$, let
$$
F(x):=\sum_{i}a_{i1}f_{i1}(x)a_{i2}f_{i2}(x)a_{i3}\cdots a_{is}f_{is}(x)a_{i\,{s+1}}
$$
for all $x\in D$. Applying the standard linearizing argument to $F(x)$, we define
$$
F^{(1)}(x_1):=F(x_1),\ \ F^{(2)}(x_1, x_2):=F^{(1)}(x_1+x_2)-F^{(1)}(x_1)-F^{(1)}(x_2).
$$
In general, let
 \begin{eqnarray*}
   && F^{(k+1)}(x_1,\cdots, x_k, x_{k+1})\\
&:=&F^{(k)}(x_1,\cdots, x_{k-1}, x_k+x_{k+1})-F^{(k)}(x_1,\cdots, x_{k-1}, x_k)-F^{(k)}(x_1,\cdots, x_{k-1}, x_{k+1})
\end{eqnarray*}
for $k\geq1$. It is well-known that
$$
F^{(s)}(x_1,\cdots, x_{s-1}, x_s)=\sum_{i}\sum_{\sigma\in \text{\rm Sym}(s)}a_{i1}f_{i1}(x_{\sigma(1)})a_{i2}f_{i2}(x_{\sigma(2)})a_{i3}\cdots a_{is}f_{is}(x_{\sigma(s)})a_{i\,{s+1}}
$$
for all $x_i\in D$, where $\text{\rm Sym}(s)$ is the symmetric group of $\{1, 2,\ldots,s\}$.
If $k>s$ we have
$$
F^{(k)}(x_1,\cdots, x_{k-1}, x_k)=0
$$
for all $x_i\in D$. The following is well-known.

\begin{remark}\label{remark1}
Let $D$ be a division ring, and let $0\ne G(X)\in D\{X\}$.

(i)\ If $\deg\, G(X)>1$, then
$G(X+Y)-G(X)-G(Y)\ne 0$ in $D\{X, Y\}$.

(ii)\ If $G(X+Y)-G(X)-G(Y)\ne 0$
in $D\{X, Y\}$ and $G(0)=0$, then $\deg\, G(X)>1$.

(iii)\ $G(X+Y)=G(X)+G(Y)$ if and only if  there exist finitely many $a_i, b_i\in D$ such that $G(X)=\sum_ia_iXb_i$.

(iv)\ If $s:=\deg\, G(X)\geq 1$ and $G(0)=0$, then $G^{(s)}(X_1,\cdots,X_s)$ is nonzero and multilinear in $D\{X_1,\cdots,X_s\}$.
\end{remark}

Given additive maps $f_{ij}\colon D\to D$, let
$F(x):=\sum_sF_s(x)$ for $x\in D$, where
$$
F_s(x):=\sum_{i}a_{i1s}f_{i1}(x)a_{i2s}f_{i2}(x)a_{i3s}\cdots a_{iss}f_{is}(x)a_{i\,{s+1}s}
$$
for all $x\in D$. For a positive integer $t$ we define
$$
F^{(t)}(x_1,\cdots, x_t):=\sum_s{F_s}^{(t)}(x_1,\cdots, x_t)
$$
for all $x_i\in D$. We are now ready to prove the first main theorem in the paper.

\begin{thm} \label{thm11}
  Let $D$ be a division ring, and let $f\colon D\to D$ be  an additive map
  satisfying $G(x)f(x)=H(x)$ for all $x\in D$, where $G(X), H(X)\in D\{X\}$.
 If $G(X)\ne 0$, then either $D$ is finite-dimensional over its center or $f$ is an elementary operator.
 \end{thm}

\begin{proof}
Let $s:=\deg\,G(X)$.
We first consider the case $s=0$. Then $G(X)=b$ for some $b\in D^\times$.
Thus $f(x)=b^{-1}H(x)$ for all $x\in D$. By the additivity of $f$, we get $H(0)=0$.

If $\deg\,H(X)\leq 1$, then there exist finitely many $a_i, b_i\in D$ such that $H(X)=\sum_ia_iXb_i$.
Thus $f$ is an elementary operator. Assume next that $\deg\,H(X)>1$. In view of Remark \ref{remark1},
$$
T(X, Y):=H(X+Y)-H(X)-H(Y)\ne 0
$$
in $D\{X, Y\}$. By the additivity of $f$, $T(X, Y)$ is a nontrivial GPI for $D$ and hence $D$ is a GPI-algebra.
It follows from Theorem \ref{thm12} that $[D\colon Z(D)]<\infty$.

We consider the case $s:=\deg\,G(X)\geq 1$.
Write
$$
G(X)=G_0(X)+G_1(X),
$$
where $G_1(X)$ is the homogeneous part of $G(X)$ of degree $s$. Set
$$
A(x):=G_0(x)f(x)\ \ \text{\rm and}\ \ \ B(x):=G_1(x)f(x)
$$
for all $x\in D$. Since $\deg\, G_0(X)<s$, we have
$$
A^{(s+1)}(x_1,\ldots,x_{s+1})=0
$$
and
$$
B^{(s+1)}(x_1,\ldots,x_{s+1})=\sum_{j=1}^{s+1}{G_1}^{(s)}(x_1,\cdots,\widehat{x_j},\cdots,x_{s+1})f(x_j)
$$
for all $x_i\in D$. Since $G(x)f(x)=H(x)$ for all $x\in D$, we get
\begin{eqnarray}
\sum_{j=1}^{s+1}{G_1}^{(s)}(x_1,\cdots,\widehat{x_j},\cdots,x_{s+1})f(x_j)=H^{(s+1)}(x_1,\cdots,x_{s+1})
\label{eq:100}
\end{eqnarray}
for all $x_i\in D$. In view of Remark \ref{remark1}, all ${G_1}^{(s)}(X_1,\cdots,\widehat{X_j},\cdots,X_{s+1})$'s are multilinear, and
$$
{G_1}^{(s)}(X_1,\cdots,X_{s})\ne 0
$$
since ${G_1}(X)\ne 0$ and $\deg\,{G_1}(X)=s$. Hence we can rewrite Eq.\eqref{eq:100} as
\begin{eqnarray}
&&{G_1}^{(s)}(x_1,\cdots,x_s)f(x_{s+1})\nonumber\\
&=&\sum_jb_j(x_1,\cdots,x_s)x_{s+1}c_j(x_1,\cdots,x_s)+H^{(s+1)}(x_1,\cdots,x_{s+1}),
\label{eq:101}
\end{eqnarray}
for all $x_i\in D$, where $b_j, c_j$ are generalized monomials in $x_1,\cdots,x_s, f(x_1),\cdots,f(x_s)$.\vskip4pt

Case 1:\ ${G_1}^{(s)}(x_1,\cdots,x_s)=0$ for all $x_1,\cdots,x_s\in D$. Then $D$ is a division GPI-algebra.
It follows from Theorem \ref{thm12} that $[D\colon Z(D)]<\infty$. We are done in this case.

Case 2:\ ${G_1}^{(s)}(z_1,\cdots,z_s)\ne 0$ for some $z_1,\cdots,z_s\in D$. We let
$$
\widetilde{b_j}:=b_j(z_1,\cdots,z_s)\ \ \text{\rm and}\ \ \widetilde{c_j}:=c_j(z_1,\cdots,z_s)
$$
for all $j$.
In view of Eq.\eqref{eq:101}, we have
\begin{eqnarray}
f(x_{s+1})=\sum_j{G_1}^{(s)}(z_1,\cdots,z_s)^{-1}\widetilde{b_j}x_{s+1}\widetilde{c_j}+E(x_{s+1}),
\label{eq:102}
\end{eqnarray}
for all $x_{s+1}\in D$, where
$$
E(X_{s+1}):={G_1}^{(s)}(z_1,\cdots,z_s)^{-1}H^{(s+1)}(z_1,\cdots,z_{s}, X_{s+1})\in D\{X_{s+1}\}.
$$

Suppose first that $E(X_{s+1})$ is linear in $X_{s+1}$.
Hence, by Eq.\eqref{eq:102}, $f$ is an elementary operator. Suppose next that $E(X_{s+1})$ has degree greater than $1$. By the additivity of $f$, it follows from Eq.\eqref{eq:102} that $D$ satisfies the GPI
$$
E(X_{s+1}+X_{s+2})-E(X_{s+1})-E(X_{s+2}).
$$
By Remark \ref{remark1}, $E(X_{s+1}+X_{s+2})-E(X_{s+1})-E(X_{s+2})\ne 0$.
It follows from Theorem \ref{thm12} that $D$ is finite-dimensional over its center $Z(D)$.
\end{proof}

\begin{cor} \label{cor9}
  Let $D$ be a division ring, and let $f_1,\ldots,f_s\colon D\to D$ be additive maps
  satisfying
\begin{eqnarray}
  \sum_{j=1}^sG_j(x_1,\ldots,x_s)f_j(x_j)=H(x_1,\ldots,x_s)
  \label{eq:27}
\end{eqnarray}
  for all $x_i\in D$, where $G_j(X_1,\ldots,X_s), H(X_1,\ldots,X_s) \in D\{X_1,\ldots,X_s\}$ for $j=1,\ldots,s$.
 Suppose that $[D\colon Z(D)]=\infty$. Given $j$, if $G_j(X_1,\ldots,X_s)\ne 0$, then $f_j$ is an elementary operator.
 \end{cor}

\begin{proof}
We may assume without loss of generality that $j=1$ and $G_1(X_1,\ldots,X_s)\ne 0$.
We claim that $f_1$ is an elementary operator.
 If $G_1(X_1,\ldots,X_s)$ is a GPI for $D$, it follows from Theorem \ref{thm12} that $D$ is finite-dimensional over its center, a contradiction. Thus $G_1(X_1,\ldots,X_s)$ is not a GPI for $D$. Then
$G_1(a_1,\ldots,a_s)\ne 0$ for some $a_i\in D$. By Eq.\eqref{eq:27}, we have
\begin{eqnarray}
G_1(x_1,a_2,\ldots,a_s)f(x_1)=- \sum_{j=2}^sG_j(x_1,a_2,\ldots,a_s)f_j(a_j)+H(x_1,a_2,\ldots,a_s)
  \label{eq:28}
\end{eqnarray}
  for all $x_1\in D$.
  Note that
  $$
  G_1(X_1,a_2,\ldots,a_s), - \sum_{j=2}^sG_j(X_1,a_2,\ldots,a_s)f_j(a_j)+H(X_1,a_2,\ldots,a_s)\in D\{X_1\},
  $$
and $G_1(X_1,a_2,\ldots,a_s)\ne 0$. Since $[D\colon Z(D)]=\infty$, it follows from
Theorem \ref{thm11} that $f_1$ is an elementary operator.
\end{proof}

The following extends Theorem \ref{thm11} to a more general form.

\begin{thm} \label{thm15}
  Let $D$ be a division ring, and let $f_i\colon D\to D$ be  additive maps, $1\leq i\leq s$,
  satisfying
 \begin{eqnarray}
  \sum_{i=1}^nG_i(x)f_i(x)=H(x)
   \label{eq:32}
\end{eqnarray}
   for all $x\in D$, where $0\ne G_i(X), H(X)\in D\{X\}$ for all $i$. Suppose that $[D\colon Z(D)]=\infty$.
 If $\deg\,G_i(X)\ne \deg\,G_j(X)$ for $i\ne j$, then all $f_i$ are elementary operators.
 \end{thm}

 \begin{proof}
 We proceed the proof by induction on $n$. The case $n=1$ is proved by Theorem \ref{thm11}. Suppose that $n>1$
 and the conclusion holds for $n-1$.
 Let $s_i:=\deg\,G_i(X)$ for $1\leq i\leq n$. We may assume that $s_n>s_i$ for $1\leq i\leq n-1$. Write
 $$
G_n(X)=E_0(X)+E_1(X)
$$
where $E_1(X)$ is the homogeneous part of $G_n(X)$ of degree $s_n$. Since $s _n+1>s_i+1$ for $1\leq i\leq n-1$
applying the same argument given in the proof of Theorem \ref{thm11}, we get
$$
\sum_{j=1}^{s_n+1}{E_1}^{(s_n)}(x_1,\cdots,\widehat{x_j},\cdots,x_{s_n+1})f_n(x_j)=H^{(s_n+1)}(x_1,\cdots,x_{s_n+1})
$$
for all $x_i\in D$. Note that ${E_1}^{(s_n)}(X_1,\cdots,X_{s_n})\ne 0$. It follows from Corollary \ref{cor9} that $f_n$ is an elementary operator.

Thus there exist finitely many $a_i, b_i\in D$ such that
$f_n(x)=\sum_ia_ixb_i$
for all $x\in D$. By Eq.\eqref{eq:32},
 \begin{eqnarray*}
  \sum_{i=1}^{n-1}G_i(x)f_i(x)=H(x)-G_n(x)\sum_ia_ixb_i
 \end{eqnarray*}
   for all $x\in D$. Note that $H(X)-G_n(X)\sum_ia_iXb_i\in D\{X\}$. By the inductive hypothesis, we obtain that $f_1,\ldots,f_{n-1}$ are
 elementary operators.
 \end{proof}

\section{The identity $f(x^2)=w(x)f(x)$}

The identity $f(x^2)=w(x)f(x)$ appears normally in the fifth section.

\begin{lem}\label{lem10}
Let $D$ be a division ring of characteristic not $2$, and let $f, g\colon D\to D$ be additive maps.
Suppose that
$f(x^2)=w(x)g(x)$
for all $x\in D$, where $w\colon D\to D$ is a map. Then
\begin{eqnarray*}
&&\big(2w(2x)-w(x+y)-w(x-y)-2w(x)\big)g(x)\nonumber\\
&=&\big(w(x+y)-w(x-y)-2w(y)\big)g(y)
\end{eqnarray*}
for all $x, y\in D$.
\end{lem}

\begin{proof}
Let $x, y\in D$. We compute
\begin{eqnarray}
f(xy+yx)&=&f((x+y)^2-x^2-y^2)\nonumber\\
              &=&w(x+y)g(x+y)-w(x)g(x)-w(y)g(y)\nonumber\\
              &=&\big(w(x+y)-w(x)\big)g(x)+\big(w(x+y)-w(y)\big)g(y).
              \label{eq:20}
\end{eqnarray}
Replacing $(x, y)$ by $(x+y, x-y)$ in Eq.\eqref{eq:20}, we get
\begin{eqnarray*}
&&f((x+y)(x-y)+(x-y)(x+y))\\
&=&\big(w(2x)-w(x+y)\big)g(x+y)+\big(w(2x)-w(x-y)\big)g(x-y)\\
                                          &=&\big(2w(2x)-w(x+y)-w(x-y))\big)g(x)+\big(w(x-y)-w(x+y))\big)g(y).
\end{eqnarray*}
On the other hand,
\begin{eqnarray*}
&&f((x+y)(x-y)+(x-y)(x+y))\\
&=&2f(x^2)-2f(y^2)\\
&=&2w(x)g(x)-2w(y)g(y).
\end{eqnarray*}
Comparing the above two equations, we get
\begin{eqnarray*}
&&\big(2w(2x)-w(x+y)-w(x-y)-2w(x)\big)g(x)\nonumber\\
&=&\big(w(x+y)-w(x-y)-2w(y)\big)g(y),
\end{eqnarray*}
as desired.
\end{proof}

The following gives an application of Corollary \ref{cor9}.

\begin{thm}\label{thm14}
Let $D$ be a division ring of characteristic not $2$, and let $f, g\colon D\to D$ be nonzero additive maps.
Suppose that
$f(x^2)=w(x)g(x)$
for all $x\in D$, where $0\ne w(X)\in D\{X\}$. If $\deg\,w(X)>1$, then either $D$ is finite-dimensional over $Z(D)$ or $g$ is an elementary operator.
\end{thm}

\begin{proof}
Assume that $\deg\,w(X)>1$ and $[D\colon Z(D)]=\infty$.
By Remark \ref{remark1}, we have
$$
w(X+Y)-w(X)-w(Y)\in D\{X, Y\}\setminus \{0\}.
$$
It follows from Lemma \ref{lem10} that
\begin{eqnarray*}
&&\big(2w(2x)-w(x+y)-w(x-y)-2w(x)\big)g(x)\nonumber\\
&=&\big(w(x+y)-w(x-y)-2w(y)\big)g(y).
\end{eqnarray*}
We claim that $g$ is an elementary operator.
Otherwise, it follows from Corollary \ref{cor9} that
\begin{eqnarray}
2w(2X)-w(X+Y)-w(X-Y)-2w(X)=0
\label{eq:24}
\end{eqnarray}
and
\begin{eqnarray}
w(X+Y)-w(X-Y)-2w(Y)=0.
\label{eq:25}
\end{eqnarray}
Replacing $X$ by $0$ in Eq.\eqref{eq:24}, we get $w(-Y)=-w(Y)$, which implies $w(0)=0$ since ${\text{\rm char}}\,D\ne 2$.
Replacing $Y$ by $X$ in Eq.\eqref{eq:24}, we get $2w(2X)-w(2X)-w(0)-2w(X)=0$ and so $w(2X)=2w(X)$. Hence we can rewrite Eq.\eqref{eq:24} as
\begin{eqnarray}
w(X+Y)+w(X-Y)=2w(X).
\label{eq:26}
\end{eqnarray}
It follows from Eq.\eqref{eq:25} and Eq.\eqref{eq:26} that $w(X+Y)=w(X)+w(Y)$, a contradiction.
\end{proof}

The following is a consequence of \cite[Theorem 2(a)]{Martindale1969}.

\begin{lem}\label{lem13}
Let $D$ be a division ring, and let $\{a_1,\ldots,a_s\}$ and $\{b_1,\ldots,b_s\}$ be two independent subsets of $D$ over $Z(D)$.
Then $\sum_{i=1}^sa_iXb_i\ne 0$ in $D\{X\}$.
\end{lem}

\begin{lem}\label{lem12}
Let $D$ be a division ring, and let $w(X)\in D\{X\}$ with $\deg\,w(X)>1$.
Suppose that $\{a_1,\ldots,a_s\}$ and $\{b_1,\ldots,b_s\}$ are two independent subsets of $D$ over $Z(D)$.
Then
$$
\sum_{i=1}^sa_iX^2b_i-\sum_{i=1}^sw(X)a_iXb_i\ne 0.
$$
\end{lem}

\begin{proof}
Let $\ell:=\deg\,w(X)>1$. Write
$w(X)=\sum_{j=0}^\ell w_j(X)$,
where $w_j(X)$ is the homogeneous part of $w(X)$ of degree $j$, $j=0,\ldots,\ell$.  Thus $w_\ell(X)\ne 0$.
Suppose on the contrary that
$$
\sum_{i=1}^sa_iX^2b_i-\sum_{i=1}^sw(X)a_iXb_i=0.
$$
Then
$$
\sum_{i=1}^sa_iX^2b_i-\sum_{j=0}^\ell \big(w_j(X)\sum_{i=1}^sa_iXb_i\big)=0.
$$
Thus, by $\ell>1$, $w_\ell(X)\sum_{i=1}^sa_iXb_i=0$. Since $D\{X\}$ is a domain (see \cite[Corollary, p.379]{Cohn1968}),
we have $\sum_{i=1}^sa_iXb_i=0$. This is a contradiction since $\sum_{i=1}^sa_iXb_i\ne 0$ (see Lemma \ref{lem13}).
\end{proof}

The following will be used later.

\begin{cor}\label{cor3}
Let $D$ be a division ring of characteristic not $2$, and let $f\colon D\to D$ be a nonzero additive map.
Suppose that
$f(x^2)=w(x)f(x)$
for all $x\in D$, where $0\ne w(X)\in D\{X\}$. If $\deg\, w(X)>1$, then $D$ is finite-dimensional over $Z(D)$.
\end{cor}

\begin{proof}
Assume that $\deg\, w(X)>1$. By Theorem \ref{thm14}, either $D$ is finite-dimensional over its center $Z(D)$ or $f$ is an elementary operator.
We assume that $f$ is an elementary operator.
There exist finitely many $a_i, b_i\in D$, $i=1,\ldots,s$, such that $f(x)=\sum_{i=1}^sa_ixb_i$ for all $x\in D$. We can choose $s$ to be minimal. Then $\{a_1,\ldots,a_s\}$ and $\{b_1,\ldots,b_s\}$ are two independent sets over $Z(D)$. Since
$f(x^2) =w(x)f(x)$ for all $x \in D$, we get
$\sum_{i=1}^sa_ix^2b_i=\sum_{i=1}^sw(x)a_ixb_i$
for all $x\in D$. Since $\deg\,w(X)\geq 2$, it follows from Lemma \ref{lem12} that
$$
\sum_{i=1}^sa_iX^2b_i-\sum_{i=1}^sw(X)a_iXb_i
$$
is a nontrivial GPI for $D$. Applying Theorem \ref{thm12}, we get $[D\colon Z(D)]<\infty$.
\end{proof}

\section{A preliminary result}
In this section we prove a preliminary result on the identity $f(x) = x^n g(x^{-1})$.
  Let $R$ be a ring and $I$ be an ideal of $R$. For $x \in R$, we write $\overbar{x} := x + I \in R/I$.

  \begin{thm} \label{thm2}
    Let $R$ be an algebra over a field $F$ and $n \ne 2$ be an integer. Suppose that $f, g \colon R \rightarrow R$ are additive maps satisfying
    \begin{eqnarray} \label{eq:1}
      f(x) = x^n g(x^{-1})
    \end{eqnarray}
    for all $x \in R^{\times}$. Then $f = g=0$ on $R^{\times}$ except when ${\text{\rm char}}\,F = p >0$ with $p-1 \mid n-2$.
  \end{thm}

  \begin{proof}
    If ${\text{\rm char}}\,F =0$, then we let $k:=2$. If ${\text{\rm char}}\,F = p >0$ with $p-1 \nmid n-2$, then we choose a positive integer $k$ such that $\overbar{k}$ generates the cyclic multiplicative group $(\mathbb{Z} / p \mathbb{Z})^{\times}$. Since $p-1 \nmid n-2$, we have $k^{n-2}-1 \ne 0$ in $F$. Both of the two cases satisfy $k(k^{n-2}-1) \ne 0$ in $F$. Replacing $x$ by $kx$ in Eq.\eqref{eq:1}, we have
    \begin{eqnarray*}
      kf(x) &=& f(kx) \\
            &=& k^n x^n g(k^{-1}x^{-1}) \\
            &=& k^{n-1} x^n g(x^{-1}) \\
            &=& k^{n-1}f(x)
    \end{eqnarray*}
    for all $x \in R^{\times}$, implying $k(k^{n-2}-1)f(x)=0$ for all $x \in R^{\times}$. Hence $f = 0$ on $R^{\times}$, and so $g = 0$ on $R^{\times}$ by Eq.\eqref{eq:1}.
  \end{proof}

 In view of Theorem \ref{thm2}, we have the following consequence, which reduces the study of Theorem \ref{thm10} to the case that
 ${\text{\rm char}}\,D = p >0$ and $p-1 \mid n-2$.

 \begin{thm} \label{thm3}
 Let $D$ be a division ring and $n \ne 2$ be an integer. Suppose that $f,g \colon D \rightarrow D$ are additive maps satisfying
 $f(x) = x^n g(x^{-1})$ for all $x \in D^{\times}$. Then $f = g = 0$ except when ${\text{\rm char}}\,D = p >0$ with $p-1 \mid n-2$.
 \end{thm}

We provide a series of examples to show that the restriction on the characteristic in Theorem \ref{thm2} is necessary.

  \begin{examp}
    \begin{enumerate} [{\rm (i)}]
      \item Let $R = \mathbb{Z}/2\mathbb{Z}$ and $f(x)=g(x)=x$ for all $x \in R$. Then it is easy to see that $f$ and $g$ satisfy Eq.\eqref{eq:1} for any integer $n$.
      \item Let $R = \mathbb{Z}/p\mathbb{Z}$, where $p$ is a prime number, and $f(x)=g(x)=x$ for all $x \in R$. Then $x^p = x$ for all $x \in R$. Let $k \in \mathbb{Z}$ and $n = (p-1)k+2$. Then
          $$
          x^n g(x^{-1}) = x^{pk-k+2} x^{-1} = x = f(x)
          $$
          for all $x \ne 0$. Here $n$ ranges over all integers such that $p-1 \mid n-2$.
      \item Let $R$ be a finite field of order $p^k$, where $p$ is a prime number and $k$ is a positive integer, and $f(x)=x$, $g(x)=x^{p^{\ell}}$ for all $x \in R$, where $\ell$ is a non-negative integer. Then $x^{p^k} = x$ for all $x \in R$. If $n = p^{\ell}+p^k$, then
          $$
          x^n g(x^{-1}) = x^{p^{\ell}+p^k} x^{-p^{\ell}} = x^{p^k} = x = f(x)
          $$
          for all $x \ne 0$.
      \item Let $R$ be an infinite field of characteristic $p>0$ (e.g., the algebraic closure of $\mathbb{Z}/p\mathbb{Z}$), and $f(x)=\alpha x^{p^{\ell}} + \beta x^{p^{m}}$, $g(x)=\alpha x^{p^{m}} + \beta x^{p^{\ell}}$ for all $x \in R$, where $\ell,m$ are positive integers and $\alpha, \beta \in R$. If $n = p^{\ell}+p^{m}$, then
          $$
          x^n g(x^{-1}) = x^{p^{\ell}+p^{m}} (\alpha x^{-p^{m}} + \beta x^{-p^{\ell}}) = \alpha x^{p^{\ell}} + \beta x^{p^{m}} = f(x)
          $$
          for all $x \in R^{\times}$.
      \item Let $F$ be a field of characteristic $p>0$ and $R = F\{X,Y\}$. Then $R^{\times} = F$. Let $n$ be an integer and $f,g \colon R \rightarrow R$ be nonzero additive maps such that $f|_F$, $g|_F$ and $n$ are as in each of {\rm (i)--(iv)}. Then
          $$
          f(x) = x^n g(x^{-1})
          $$
          for all $x \in F = R^{\times}$. In this case, $R$ is a noncommutative domain, which is not a division ring.
              \end{enumerate}
              \label{example1}
  \end{examp}

\section{The identity $f(x) = x^n g(x^{-1})$}

The aim of this section is to solve  Eq.\eqref{eq:17}. We extend Catalano and Merch\'{a}n's result in 2023  to get a complete solution.

 \begin{thm} \label{thm10}
 Let $D$ be a noncommutative division ring with ${\text{\rm char}}\,D \ne 2$. Let $f, g \colon D \rightarrow D$ be additive maps satisfying
  $$
  f(x) = x^n g(x^{-1})
  $$
  for all $x \in D^{\times}$, where $n \ne 2$ is a positive integer. Then $f = 0$ and $g = 0$.
\end{thm}

\begin{remark}\label{remark2}
Theorem \ref{thm10} is in general not true if $D$ is only a noncommutative domain, not a division ring. Example \ref{example1} (v) provides such an example.
\end{remark}

Let $D$ be a division ring of characteristic not $2$. If $n = 1$ in Eq.\eqref{eq:17} , then $f = g = 0$ by Theorem \ref{thm3}. The case $n = 2$  has been completely characterized by Dar and Jing \cite{Dar2023} (see Theorem \ref{thm13}). According to Theorem \ref{thm3}, in this section we always assume the following:\ \
$$
(\ddag)\ \ \ \ \ \ \ \ \ \ \ \ \ {\text{\rm char}}\,D = p > 2,\ n>2\ \text{\rm and}\ p-1 \mid n-2.
$$

Therefore, to study Theorem \ref{thm10} we divide our arguments into the following two cases:\vskip6pt

{\bf Case I}:\ $n = p^{\ell}k$, where $\ell \geq 0$, $k>1$ and $\text{\rm gcd}(p,k) = 1$, and $k-1$ is not a non-negative power of $p$;

{\bf Case II}:\ $n = p^{\ell+m}+p^{\ell}$ for some integers $\ell \geq 0$ and $m \geq 0$ with $(\ell,m)\ne (0,0)$.\vskip6pt

The following observation plays a key role for Case I.

\begin{lem} \label{lem3}
  Let $k>1$ be a positive integer and $p>2$ be a prime number. Suppose that $\text{\rm gcd}(p,k) = 1$ and $k-1$ is not a non-negative power of $p$. Let $m \geq 0$ be the largest integer such that $p^m \mid k-1$. Then the following hold:

  (i)\ \ $p \nmid \binom{k}{p^m+1}$;

(ii)\ \ $P(X):= 2\sum\limits_{ \substack{ t=2 \\ t \colon even}}^{k-2} \binom{k}{t} X^{p^{\ell}t} \in (\mathbb{Z}/p\mathbb{Z})[X]\setminus \{0\}$.
\end{lem}

\begin{proof}
 (i)\ Note that $1 = p^0$, and so $k-1 > 1$. Thus $k > 2$. We proceed the proof by induction on $m$.

 We first consider the case $m = 0$. Then $p \nmid k-1$, and it follows from $\text{\rm gcd}(p,k) = 1$ that $p\nmid k(k-1)$. So
  $  p \nmid \binom{k}{2} = \binom{k}{p^0+1}$.
  If $m = 1$, then $p^2 \nmid k-1$, and so obviously
  $$
  p \nmid \binom{k}{p+1}.
  $$
  Assume that $m>1$ and the conclusion holds for $m-1$. Let $k':=\dfrac{k-1}{p}+1$. Then
  $$
  p^{m-1} \mid \dfrac{k-1}{p} =k'-1
  $$
  and so $\text{\rm gcd}(p, k') = 1$. Moreover, we have
  $$
  p^m \nmid (\dfrac{k-1}{p}+1)-1=k'-1.
  $$
 Since $k-1$ is not a non-negative power of $p$, we have
 $ k'-1=\dfrac{k-1}{p}$,
 which  is not a non-negative power of $p$.  We compute
  \begin{eqnarray*}
    \binom{k'}{p^{m-1}+1}&=& \dfrac{(\dfrac{k-1}{p}+1)\dfrac{k-1}{p}(\dfrac{k-1}{p}-1)\cdots (\dfrac{k-1}{p}-p^{m-1}+1)}{(p^{m-1}+1)p^{m-1}(p^{m-1}-1)\cdots 1}  \\
                                                         &=&\dfrac{\prod_{j=-1}^{p^{m-1}-1}\big((k-1)-jp\big)}{\prod_{j=-1}^{p^{m-1}-1}(p^{m}-jp)}.\\                                                    \end{eqnarray*}
  By the inductive hypothesis, $p\nmid \binom{k'}{p^{m-1}+1}$ and so
  \begin{eqnarray} \label{eq:6}
    p \nmid \dfrac{\prod_{j=-1}^{p^{m-1}-1}\big((k-1)-jp\big)}{\prod_{j=-1}^{p^{m-1}-1}(p^{m}-jp)}.
  \end{eqnarray}
  Consider the following equation
  \begin{eqnarray} \label{eq:7}
    \binom{k}{p^m+1} = \dfrac{k(k-1)\cdots (k-p^m)}{(p^m+1)p^m(p^m-1)\cdots 2\cdot 1}
    =\dfrac{\prod_{j=-1}^{p^{m}-1}\big((k-1)-j\big)}{\prod_{j=-1}^{p^{m}-1}\big(p^{m}-j\big)}.
  \end{eqnarray}
  Note that $\prod_{j=-1}^{p^{m-1}-1}\big((k-1)-jp\big)$ and $\prod_{j=-1}^{p^{m}-1}\big((k-1)-j\big)$ have the same multiplicity of the prime factor $p$. Also, $\prod_{j=-1}^{p^{m-1}-1} (p^{m}-jp)$ and $\prod_{j=-1}^{p^{m}-1}\big(p^{m}-j\big)$ have the same multiplicity of the prime factor $p$.
 By  Eq.\eqref{eq:6} and Eq.\eqref{eq:7}, we get
  $p \nmid \binom{k}{p^m+1}$,
  as desired.

  (ii)\ \ Since $m \geq 0$ is the largest integer such that $p^m \mid k-1$. In this case, $p^m+1$ is even and $p^m+1 \leq k-2$ by $p^m \ne k-1$. By (i),  $p \nmid \binom{k}{p^m+1}$ and so $P(X) \ne 0$ in $(\mathbb{Z}/p\mathbb{Z})[X]$.
\end{proof}

We also need one more lemma.

\begin{lem} (\cite[Lemma 4.1(a)]{Catalano2023})
Let $D$ be a division ring, and let $n\geq 2$ be an integer. Let $f, g \colon D \rightarrow D$ be additive maps satisfying
$f(x) = x^n g(x^{-1})$ for all $x \in D^{\times}$. Then
$$
f(b)=(-(1-b)^n + b^n +1)f(1)-g(b)
$$
for all $b\in D$.
\label{lem14}
\end{lem}

\begin{lem} \label{lem1}
Let $D$ be a noncommutative division ring satisfying  ({$\ddag$}) and Case (I). Let $f,g \colon D \rightarrow D$ be additive maps satisfying
 $f(x) = x^n g(x^{-1})$
  for all $x \in D^{\times}$. Then $f(1) = 0$ and $f = -g$.
\end{lem}

\begin{proof}
  Since $p$ is odd and $p-1 \mid n-2$, $n$ must be even, and so is $k$. Also, it follows from $ k-1 \ne p^0  = 1$ that $k>2$. By Lemma \ref{lem14},
  \begin{eqnarray} \label{eq:2}
    (f+g)(a) = (-(1-a)^n + a^n +1) f(1)
  \end{eqnarray}
  for all $a \in D$. Replacing $a$ by $-a$ in Eq.\eqref{eq:2}, we have
  \begin{eqnarray} \label{eq:3}
    -(f+g)(a) = (-(1+a)^n + a^n +1) f(1)
  \end{eqnarray}
  for all $a \in D$. Comparing Eq.\eqref{eq:2} and Eq.\eqref{eq:3}, we have
  \begin{eqnarray} \label{eq:4}
    0 = ((1+a)^n + (1-a)^n - 2a^n -2) f(1)
  \end{eqnarray}
  for all $a \in D$. Since $k > 2$, the equation Eq.\eqref{eq:4} becomes
  \begin{eqnarray*}
    2\Big( \sum_{\substack{ t=2 \\ t \colon even}}^{k-2} \binom{k}{t} a^{p^{\ell}t} \Big) f(1) = 0
  \end{eqnarray*}
  for all $a \in D$. Set $P(X):= 2\sum\limits_{ \substack{ t=2 \\ t \colon even}}^{k-2} \binom{k}{t} X^{p^{\ell}t} \in (\mathbb{Z}/p\mathbb{Z})[X]$. Then
  \begin{eqnarray} \label{eq:5}
    P(a) f(1) = 0
  \end{eqnarray}
  for all $a \in D$.  It follows from Lemma \ref{lem3} (ii) that $P(X) \ne 0$ in $(\mathbb{Z}/p\mathbb{Z})[X]$.
   Since $D$ is noncommutative, according to Jacobson's theorem \cite[Theorem 2, p. 183]{Jacobson1964} (or see \cite[Theorem 13.11]{Lam2001}), there exists $a \in D$ such that $P(a) \ne 0$.
Hence it follows from Eq.\eqref{eq:5} that $f(1) = 0$, and thus $f = -g$ by Eq.\eqref{eq:2}.
\end{proof}

\begin{lem} \label{lem2}
Let $D$ be a noncommutative division ring satisfying  ({$\ddag$}) and Case (II). Let $f,g \colon D \rightarrow D$ be additive maps satisfying
$f(x) = x^n g(x^{-1})$
for all $x \in D^{\times}$.  Then $f(1) = 0$ and $f = -g$.
\end{lem}

\begin{proof}
  It follows from Lemma \ref{lem14} that
  $$
  (f+g)(a) = ( a^{p^{\ell}} + a^{p^{\ell+m}} )f(1)
  $$
  for all $a \in D$. Assume that $f(1) \ne 0$. Then the map $a \mapsto a^{p^{\ell}} + a^{p^{\ell+m}}$ is additive. Let
  \begin{eqnarray*}
  &&Q(X,Y)\\
  &:=& ((X+Y)^{p^{\ell}} + (X+Y)^{p^{\ell+m}}) - (X^{p^{\ell}} + X^{p^{\ell+m}} ) - (Y^{p^{\ell}} + Y^{p^{\ell+m}}) \in \mathbb{Z}\{X,Y\}.
\end{eqnarray*}
  Then $Q(X,Y)\ne 0$ and has coefficients $\pm 1$, implying that $D$ is a PI-ring. It follows from Posner's theorem \cite{Posner1960} that $D$ is finite-dimensional over its center $Z(D)$. Since $D$ is noncommutative, $Z(D)$ is an infinite field.
Let $F$ be a maximal subfield of $D$ containing $Z(D)$. Then $D \otimes_{Z(D)} F \cong \M_{r}(F)$, where $r = \sqrt{\dim_{Z(D)} \, D} >1$. It is well-known that $D$ and $\M_{r}(F)$ satisfy the same polynomial identities (PIs for short) (see, for example, \cite[Corollary, p.64]{Jacobson1975}). So $Q(X,Y)$ is also a PI for $\M_{r}(F)$. However,
  \begin{eqnarray*}
    Q(e_{11},e_{12}) &=& ((e_{11}+e_{12})^{p^{\ell}} + (e_{11}+e_{12})^{p^{\ell+m}}) - (e_{11}^{p^{\ell}} + e_{11}^{p^{\ell+m}} ) - (e_{12}^{p^{\ell}} + e_{12}^{p^{\ell+m}}) \\
                     &=&  2e_{12} \ne 0,
  \end{eqnarray*}
  a contradiction. Hence $f(1) = 0$ and so $f = -g$.
\end{proof}

By Theorem \ref{thm3}, Lemmas \ref{lem1} and \ref{lem2}, we have the following.

\begin{pro} \label{pro8}
  Let $D$ be a noncommutative division ring with ${\text{\rm char}}\,D \ne 2$. Let $n > 2$ be a positive integer and $f,g \colon D \rightarrow D$ be additive maps satisfying $f(x) = x^n g(x^{-1})$
    for all $x \in D^{\times}$. Then $f = -g$.
\end{pro}

From now on, we always make the following assumptions.\vskip6pt

{\it Let $D$ be a noncommutative division ring satisfying  ({$\ddag$}). Let $f \colon D \rightarrow D$ be an additive map satisfying
 $f(x) = -x^n f(x^{-1})$  for all $x \in D^{\times}$. }\vskip6pt

 The aim is to prove $f=0$. For $b \in D$, let $C_D (b):= \{ a \in D \mid \, ab=ba  \}$, the centralizer of $b$ in $D$.

\begin{lem} \label{lem4}
   Let $P(X):= (1+X)^n + (1-X)^n - 2X^n - 2 \in (\mathbb{Z}/p\mathbb{Z})[X]$. The following hold:

 (i)\  $f(x^2)= ((1+x)^n -x^n -1 )f(x)$
  for all $x\in D$;

(ii)\ $P(ab)f(a) = 0$ for all $a,b \in D$ with $ab=ba$;

(iii)\  Given $d \in D$, if $P(d) \ne 0$ then $f(C_D (d)) = 0$;

(iv)\ If $P(Z(D)) \ne 0$, then $f = 0$.
 \end{lem}

\begin{proof}
 Let $a,b \in D$ with $ab \ne 0,1$. Then, by Hua's identity,
 \begin{eqnarray*}
   a-aba = (a^{-1}+(b^{-1}-a)^{-1})^{-1}.
 \end{eqnarray*}
 Thus
 \begin{eqnarray} \label{eq:8}
   f(a-aba) &=& f((a^{-1}+(b^{-1}-a)^{-1})^{-1}) \nonumber \\
            &=& -(a^{-1}+(b^{-1}-a)^{-1})^{-n} f(a^{-1}+(b^{-1}-a)^{-1}) \nonumber \\
            &=& -(a^{-1}+(b^{-1}-a)^{-1})^{-n} (f(a^{-1}) + f((b^{-1}-a)^{-1})  \\
            &=& -(a-aba)^n (-a^{-n}f(a)-(b^{-1}-a)^{-n}f(b^{-1}-a)) \nonumber \\
            &=& (a-aba)^n (a^{-n} - (b^{-1}-a)^{-n} )f(a) + (a-aba)^n (b^{-1}-a)^{-n}f(b^{-1}) \nonumber \\
            &=& (a-aba)^n (a^{-n} - (b^{-1}-a)^{-n} )f(a) - (a-aba)^n (b^{-1}-a)^{-n} b^{-n} f(b) \nonumber
 \end{eqnarray}
 for all $a,b \in D$ with $ab \ne 0,1$. Suppose $ab = ba \ne 0,1$. Then Eq.\eqref{eq:8} becomes
 \begin{eqnarray*}
   f(a-a^2b) &=& ((1-ab)^n -(ab)^n ) f(a) - a^n f(b),
 \end{eqnarray*}
 implying
 \begin{eqnarray} \label{eq:9}
   f(a^2b) &=& (-(1-ab)^n +(ab)^n +1 ) f(a) + a^n f(b),
 \end{eqnarray}
 for all $a,b \in D$ with $ab=ba$.   Since $p$ is odd and $p-1 \mid n-2$, $n$ must be even.
  Replacing $a$ by $-a$ in Eq.\eqref{eq:9}, we have
 \begin{eqnarray} \label{eq:10}
  f(a^2b) &=& ((1+ab)^n -(ab)^n -1 ) f(a) + a^n f(b),
 \end{eqnarray}
 for all $a,b \in D$ with $ab=ba$. Replacing $b$ by $1$ in Eq.\eqref{eq:10}, we get
$$
f(a^2) = ((1+a)^n -a^n -1 )f(a).
$$
This proves (i). Comparing Eq.\eqref{eq:9} and Eq.\eqref{eq:10}, we have
$P(ab)f(a) = 0$ for all $a,b \in D$ with $ab=ba$. Thus (ii) is proved.

 Obviously, (iv) follows from (iii). To prove (iii), let $d \in D$ be such that $P(d) \ne 0$. Let $a \in C_D (d)^{\times}$ and $b = a^{-1} d$. Then $ab = ba = d$.
 By (ii) we have
 $$
    P(d)f(a)=P(ab)f(a) =0,
$$
 implying $f(a) = 0$. Hence $f(C_D (d)) = 0$, as asserted.
\end{proof}

\begin{lem} \label{lem5}
   If Case (I) holds, then $f = 0$.
\end{lem}

\begin{proof}
Suppose on the contrary that $f\ne 0$.
  Let
  $$
  P(X):= (1+X)^n + (1-X)^n - 2X^n - 2 \in (\mathbb{Z}/p\mathbb{Z})[X].
   $$
   By Lemma \ref{lem3} (ii), $P(X) \ne 0$ in $(\mathbb{Z}/p\mathbb{Z})[X]$. By Lemma \ref{lem4} (iv), it follows that $P(Z(D)) = 0$.
  \begin{enumerate}  [\bfseries {\text Case} 1:]
    \item $P(d) \ne 0$ for any $d \in D \setminus Z(D)$. By Lemma \ref{lem4}, $f(C_D (d)) = 0$ for all $d \in D \setminus Z(D)$. In particular, $f(Z(D)) = 0$ and $f(d) = 0$ for all $d \in D \setminus Z(D)$, implying $f = 0$.
    \item $P(d) = 0$ for some $d \in D \setminus Z(D)$.
  Let $w\in D \setminus Z(D)$.
    By \cite[Theorem 13.10]{Lam2001}, $C_D (w)$ is an infinite division ring. It follows from \cite[Theorem 16.7]{Lam2001} that $P(a) \ne 0$ for some $a \in C_D (w)$. According to Lemma \ref{lem4} (iii), $f(C_D (a)) = 0$. In particular, $f(w) = 0$ and $f(Z(D)) = 0$. Hence $f = 0$.
  \end{enumerate}
  Therefore we conclude that $f = 0$.
\end{proof}

We next study Case (II). The following is a key observation.

\begin{lem} \label{lem6}
Assume that Case (II) holds. For $a \in D$, if $C_D (a)$ is noncommutative, then $f(a) = 0$ and
          $2f(ab) = (a^{p^{\ell+m}} + a^{p^{\ell}})f(b)$ for all $b \in C_D (a)$.
\end{lem}

\begin{proof}
  In this case, the equation Eq.\eqref{eq:10} becomes
  \begin{eqnarray} \label{eq:11}
    f(a^2b) &=& (a^{p^{\ell+m}}b^{p^{\ell+m}} +a^{p^{\ell}}b^{p^{\ell}} ) f(a) + a^n f(b),
  \end{eqnarray}
  for all $a,b \in D$ with $ab=ba$.

  Let $a \in D$ be such that $C_D (a)$ is noncommutative. Suppose that $f(a) \ne 0$. Then, by Eq.\eqref{eq:11},
  $$
  b \mapsto a^{p^{\ell+m}}b^{p^{\ell+m}} +a^{p^{\ell}}b^{p^{\ell}}
  $$
  is an additive map on $C_D (a)$. Let
\begin{eqnarray*}
    Q(X,Y) &:=& (a^{p^{\ell+m}}(X+Y)^{p^{\ell+m}} +a^{p^{\ell}}(X+Y)^{p^{\ell}}) \\
            && - (a^{p^{\ell+m}}X^{p^{\ell+m}} +a^{p^{\ell}}X^{p^{\ell}}) - (a^{p^{\ell+m}}Y^{p^{\ell+m}} +a^{p^{\ell}}Y^{p^{\ell}}) \in Z(C_D (a))\{X,Y \}.
 \end{eqnarray*}
  Then $Q(X,Y)$ is a nontrivial GPI for $C_D (a)$. It follows from Theorem \ref{thm12} that $C_D (a)$ is finite-dimensional over $Z(C_D (a))$.
  Since $C_D (a)$ is noncommutative, $Z(C_D (a))$ is an infinite field.

  Let $F$ be a maximal subfield of $C_D (a)$ containing $Z(C_D (a))$. Then $C_D (a) \otimes_{Z(C_D (a))} F \cong \M_{r}(F)$, where $r = \sqrt{\dim_{Z(C_D (a))} \, C_D (a)} >1$. It is well-known that $C_D (a)$ and $\M_{r}(F)$ satisfy the same PIs (see, for example, \cite[Corollary, p.64]{Jacobson1975}). So $Q(X,Y)$ is a PI for $\M_{r}(F)$. Since $C_D (a)$ is a noncommutative finite-dimensional division algebra over $Z(C_D (a))$, $Z(C_D (a))$ is not algebraic over $\mathbb{Z}/p\mathbb{Z}$ by Jacobson's theorem \cite[Theorem 2, p. 183]{Jacobson1964} (or see \cite[Theorem 13.11]{Lam2001}). So we can always find an element $\beta \in Z(C_D (a))$ such that $a^{p^{\ell+m}} \beta^{p^{\ell+m}}+ a^{p^{\ell}} \beta^{p^{\ell}} \ne 0$. However,
  \begin{eqnarray*}
    Q(\beta e_{11},\beta e_{12}) &=& (a^{p^{\ell+m}}(\beta e_{11}+\beta e_{12})^{p^{\ell+m}} + a^{p^{\ell}}(\beta e_{11}+\beta e_{12})^{p^{\ell}}) \\
                      &&- (a^{p^{\ell+m}} (\beta e_{11})^{p^{\ell+m}} + a^{p^{\ell}} (\beta e_{11})^{p^{\ell}} ) - (a^{p^{\ell+m}} (\beta e_{12})^{p^{\ell+m}} + a^{p^{\ell}} (\beta e_{12})^{p^{\ell}}) \\
                     &=&  (a^{p^{\ell+m}} \beta^{p^{\ell+m}}+ a^{p^{\ell}} \beta^{p^{\ell}})e_{12} \ne 0,
  \end{eqnarray*}
  a contradiction. Hence $f(a) = 0$. In this case, it follows from Eq.\eqref{eq:11} that
    \begin{eqnarray}
    f(a^2b) = a^{p^{\ell+m}} a^{p^{\ell}} f(b)
    \label{eq:29}
  \end{eqnarray}
  for all $b \in C_D (a)$. Note that $C_D(a+1)=C_D(a)$.
  Replacing $a$ by $a+1$ in Eq.\eqref{eq:29}, we have
  $$
  2f(ab) = (a^{p^{\ell+m}} + a^{p^{\ell}})f(b)
  $$
    for all $b \in C_D (a)$.
\end{proof}

We first deal with the case $m = 0$.

\begin{lem} \label{lem9}
If $n = 2p^{\ell}$ for some integer $\ell > 0$ and $|Z(D)| \geq n-1$, then $f = 0$.
\end{lem}

\begin{proof}
  In this case,
 $f(a) = -a^{2p^{\ell}} f(a^{-1})$  for all $a \in D^{\times}$, and it follows from Eq.\eqref{eq:11} that
  \begin{eqnarray} \label{eq:12}
    f(a^2) &=& 2a^{p^{\ell}} f(a)
  \end{eqnarray}
  for all $a \in D$. Also, by Lemma \ref{lem6}, we have
  \begin{eqnarray} \label{eq:15}
  f(\alpha b) = \alpha^{p^{\ell}} f(b)
  \end{eqnarray}
  for all $\alpha \in Z(D)$ and $b \in D$.

  Note that $2^{p^\ell}\equiv 2$\, mod $p$. Applying Lemma \ref{lem10} to Eq.\eqref{eq:12}, we get
    \begin{eqnarray} \label{eq:14}
    ( 2a^{p^{\ell}} - (a+b)^{p^{\ell}} - (a-b)^{p^{\ell}} )f(a) &=& (  (a+b)^{p^{\ell}} - (a-b)^{p^{\ell}} - 2b^{p^{\ell}} ) f(b)
  \end{eqnarray}
  for all $a,b \in D$. Write
  \begin{eqnarray*}
    (X+Y)^{p^{\ell}} &=&\sum_{i=0}^{p^{\ell}}P_i(X,Y)
  \end{eqnarray*}
  where $P_i(X,Y) \in (\mathbb{Z}/p\mathbb{Z})\{ X,Y \}$ is homogeneous in $Y$ of degree $i$. Then Eq.\eqref{eq:14} becomes
 \begin{eqnarray}
\sum_{i=1}^{{p^{\ell}-1}\over 2}P_{2i}(a, b)f(a) =\sum_{i=1}^{{p^{\ell}-1}\over 2}P_{2i-1}(a, b)f(b)
\label{eq:30}
\end{eqnarray}
  for all $a,b \in D$. Replacing $b$ by $\alpha b$ in Eq.\eqref{eq:30} and applying Eq.\eqref{eq:15}, we get
  \begin{eqnarray}
\sum_{i=1}^{{p^{\ell}-1}\over 2}\alpha^{2i}P_{2i}(a, b)f(a) = \sum_{i=1}^{{p^{\ell}-1}\over 2}\alpha^{2i-1+p^{\ell}}P_{2i-1}(a, b)f(b)
\label{eq:31}
\end{eqnarray}
  for all $a,b \in D$ and $\alpha\in Z(D)$.
  Note that $|Z(D)| \geq n-1 = 2p^{\ell}-1$. Applying the standard Vandermonde argument to solve Eq.\eqref{eq:31}, we have
  \begin{eqnarray*}
  P_{2i}(a, b)f(a) =0=P_{2i-1}(a, b)f(b)
  \end{eqnarray*}
    for all $a,b \in D$ for $1\leq i\leq {{p^{\ell}-1}\over 2}$.
  Let $b\in D$ be such that $f(b)\ne 0$. Then $P_1(a, b)=0$ and hence
  $$
  [a^{p^\ell}, b]=[a, P_1(a, b)]=0
  $$
  for all $a\in D$. Thus, given $a\in D$, we get
  $$
  D = C_D(a^{p^\ell})\cup \{b\in D\mid f(b)=0\}.
  $$
  Since $f\ne 0$, we get $C_D(a^{p^\ell})=D$. That is, $a^{p^\ell}\in Z(D)$ for all $a\in D$. It follows from Kaplansky's theorem \cite[Theorem]{Kaplansky1951} that $D$ is commutative, a contradiction. This proves that $f=0$.
\end{proof}

For $m > 0$, we have the following result.

\begin{lem} \label{lem7}
 If $n = p^{\ell+m}+p^{\ell}$ for some integers $\ell \geq 0$ and $m > 0$, then either $f = 0$ or $p^m$ is a positive power of $|Z(D)|$.
\end{lem}

\begin{proof}
  Let $\alpha \in Z(D)$. Then, by Lemma \ref{lem6},
  $$
  2f(\alpha b) = (\alpha^{p^{\ell+m}} + \alpha^{p^{\ell}})f(b)
  $$
  for all $b \in D$. Thus
  $$
  4f(\alpha^2 b) = (2\alpha^{2p^{\ell+m}} + 2\alpha^{2p^{\ell}})f(b)
  $$
  for all $b \in D$. On the other hand,
  \begin{eqnarray*}
    4f(\alpha^2 b) &=& 4f(\alpha (\alpha b)) \\
                   &=& 2 (\alpha^{p^{\ell+m}} + \alpha^{p^{\ell}}) f(\alpha b) \\
                   &=& (\alpha^{p^{\ell+m}} + \alpha^{p^{\ell}})^2 f(b)
  \end{eqnarray*}
  for all $b \in D$. Comparing the above two equations, we have
  $$
  (\alpha^{p^m} - \alpha)^{2p^{\ell}} f(b) = 0
  $$
  for all $\alpha \in Z(D)$ and $b \in D$, implying either $f = 0$ or $\alpha^{p^m} - \alpha = 0$ for all $\alpha \in Z(D)$. If $f \ne 0$, then $Z(D)$ is a subfield of the finite field of order $p^m$, implying $p^m$ is a positive power of $|Z(D)|$.
\end{proof}

We are ready to give the proof of Theorem \ref{thm10}.\vskip6pt

\noindent {\bf Proof of Theorem \ref{thm10}.}
In view Theorem \ref{thm3}, we may assume that $D$ satisfies ($\ddag$).
By Proposition \ref{pro8}, we may assume that $g=-f$. Note that Case (I) has been proved in Lemma \ref{lem5}.
The rest is to deal with Case (II). That is, we may assume that $n = p^{\ell+m}+p^{\ell}$ for some integers $\ell \geq 0$ and $m \geq 0$, $(\ell, m)\ne (0, 0)$.

Suppose on the contrary that $f\ne 0$. In view of Lemmas \ref{lem9} and \ref{lem7}, $Z(D)$ is a finite field.
By Lemma \ref{lem4} (i) we have
 \begin{eqnarray*}
f(x^2)&=& ((1+x)^n -x^n -1 )f(x)\\
          &=& ((1+x)^{p^{\ell+m}+p^{\ell}} -x^{p^{\ell+m}+p^{\ell}} -1 )f(x)\\
           &=& \big(x^{p^{\ell+m}}+x^{p^{\ell}}\big)f(x).
\end{eqnarray*}
Note that $p^{\ell+m}>1$. It follows from Corollary \ref{cor3} that $D$ is finite-dimensional over $Z(D)$.
Since $|Z(D)|<\infty$, $D$ is a finite division ring.  By Wedderburn's theorem, $D$ is commutative, a contradiction.
\hfill $\square$

\section{Rings additively generated by units}
  A ring (or an algebra) $R$ is said to be additively generated by units if each element of $R$ can be written as a finite sum of units in $R$. We are now ready to show our final result, which extends Theorem \ref{thm3} to algebras additively generated by units.

  \begin{thm} \label{thm5}
    Let $F$ be a field and $R$ be an $F$-algebra additively generated by units. Given an integer $n \ne 2$, suppose that $f,g \colon R \rightarrow R$ are additive maps satisfying
    \begin{eqnarray*}
      f(x) = x^n g(x^{-1})
    \end{eqnarray*}
    for all $x \in R^{\times}$. Then $f = g = 0$ except when ${\text{\rm char}}\,F = p >0$ with $p-1 \mid n-2$.
  \end{thm}

  \begin{proof}
    It follows from Theorem \ref{thm2} that $f = g = 0$ on $R^{\times}$. Let $x \in R$. Then we can write $x = u_1 + \cdots + u_m$, where $u_1, \dots, u_m \in R^{\times}$. Now we compute
    \begin{eqnarray*}
    f(x) = f(u_1 + \cdots + u_m)=f(u_1) + \cdots + f(u_m)= 0,
    \end{eqnarray*}
    implying $f=0$ and so $g=0$, as desired.
  \end{proof}

  Note that the class of rings additively generated by units contains several natural classes of rings. We state some of them here as the end of the article.

  \begin{examp}
  The following rings are additively generated by units:

    \begin{enumerate} [{\rm (i)}]
      \item The ring of all linear transformations of vector spaces over a division ring (\cite{Wolfson1953,Zelinsky1954}).
      \item The ring of all real-valued continuous functions on any completely regular Hausdorff space (\cite{Raphael1974}).
      \item Any real or complex Banach algebra (\cite{Vamos2005}).
      \item The matrix ring $\M_n (R)$, $n>1$, over any ring $R$ (\cite{Raphael1974}).
      \item The ring of all $\omega \times \omega$ row and column-finite matrices over any ring (\cite{Wang2009}).
      \item Any unit-regular ring in which $2$ is a unit (\cite{Ehrlich1968}).
      \item Any unit-regular ring in which $1$ is a sum of two units (\cite{Grover2014}).
      \item The ring of all endomorphisms of any free module of (finite or infinite) rank $>1$ over a PID (\cite{Wans1995}, or see \cite{Meehan2001}).
      \item The group ring $R[C_n]$, where $R$ is any local ring with Jacobson radical $J(R)$ such that ${\text{\rm char}}\,R/J(R) \ne 2$ and $C_n$ is any cyclic group of order $n$ (\cite{Wang2009}).
    \end{enumerate}
  \end{examp}

  We refer the reader to \cite{Srivastava2010} and related articles for more details and examples.

\end{document}